\documentclass[11pt,reqno,oneside]{amsart}

\usepackage[utf8]{inputenc}
\usepackage[T1]{fontenc}
\usepackage[top=3cm,bottom=3cm,left=3cm,right=3cm]{geometry}

\usepackage{hyperref}
\usepackage{mathtools}
\usepackage{thmtools}
\usepackage[capitalize, nameinlink]{cleveref}

\usepackage[table,dvipsnames]{xcolor}
\definecolor{red}{rgb}{1,0,0}
\hypersetup{
    colorlinks,
    linkcolor={red!50!black},
    citecolor={blue!50!black},
    urlcolor={blue!80!black}
}

\usepackage{amsfonts}
\usepackage{amssymb}
\usepackage{latexsym}
\usepackage{mathrsfs}
\usepackage{textcomp}
\usepackage{longtable}
\usepackage{booktabs}
\usepackage{graphicx}
\usepackage{setspace}
\usepackage{nicefrac}
\usepackage{enumerate}
\usepackage{wasysym}
\usepackage[normalem]{ulem}
\usepackage{upgreek}
\usepackage{paralist}
\usepackage{etoolbox}
\usepackage{mathdots}
\usepackage{wrapfig}
\usepackage{floatflt}
\usepackage{tensor} 
\usepackage{lscape}
\usepackage{stmaryrd}

\usepackage{tikz}
\usetikzlibrary{arrows}
\usetikzlibrary{matrix}
\usetikzlibrary{positioning}
\usetikzlibrary{calc}

\usepackage[all,cmtip]{xy} 
\usepackage[labelfont=bf, font={small}]{caption}[2005/07/16]

\usepackage{enumitem}

\newcommand{\xto}[1]{\xrightarrow{\phantom{a}{#1}{\phantom{a}}}}

\newcommand{\vvirg}{ , \dots , }
\newcommand{\ootimes}{ \otimes \cdots \otimes }

\newcommand{\textfrac}[2]{{\textstyle \frac{#1}{#2}}}

\newcommand{\bft}{\mathbf{t}}
\newcommand{\bfu}{\mathbf{u}}

\newcommand{\bfx}{\mathbf{x}}

\newcommand{\calD}{\mathcal{D}}
\newcommand{\calE}{\mathcal{E}}

\newcommand{\calH}{\mathcal{H}}
\newcommand{\calI}{\mathcal{I}}

\newcommand{\calO}{\mathcal{O}}

\newcommand{\calR}{\mathcal{R}}

\newcommand{\bbA}{\mathbb{A}}

\newcommand{\bbC}{\mathbb{C}}

\newcommand{\bbP}{\mathbb{P}}

\newcommand{\bbT}{\mathbb{T}}

\newcommand{\bbZ}{\mathbb{Z}}

\newcommand{\frakS}{\mathfrak{S}}

\newcommand{\bfrho}{\boldsymbol{\rho}}

\renewcommand{\phi}{\varphi}
\newcommand{\eps}{\varepsilon}

\newcommand{\dashto}{\dashrightarrow}

\renewcommand{\tilde}[1]{\widetilde{#1}}

\renewcommand{\bar}[1]{\overline{#1}}

\newcommand{\id}{\mathrm{id}}

\newcommand{\sym}{\mathrm{sym}}

\DeclareMathOperator{\mult}{mult}

\newcommand{\mon}{\mathit{mon}}

\newcommand{\smooth}{\mathrm{smooth}}

\DeclareMathAccent{\wtilde}{\mathord}{largesymbols}{"65}

\newcommand{\res}{\mathit{res}}
\newcommand{\disc}{\mathit{disc}}

\newcommand{\calDisc}{{\calD\mathit{isc}}}
\newcommand{\calRes}{{\calR\mathit{es}}}

\newcommand{\GL}{\mathrm{GL}}

\newtheorem{theorem}{Theorem}[section]

\newtheorem{lemma}[theorem]{Lemma}
\newtheorem{proposition}[theorem]{Proposition}
\newtheorem{corollary}[theorem]{Corollary}
\newtheorem{conjecture}[theorem]{Conjecture}

\theoremstyle{definition}

\theoremstyle{definition}
\newtheorem{examplex}[theorem]{Example}
\newenvironment{example}
  {\pushQED{\qed}\examplex}
  {\popQED\endexamplex}

\newtheorem{definition}[theorem]{Definition}

\newtheorem{remarkx}[theorem]{Remark}
\newenvironment{remark}
  {\pushQED{\qed}\remarkx}
  {\popQED\endremarkx}

\crefname{theorem}{Theorem}{Theorems}
\crefname{conjecture}{Conjecture}{Conjectures}
\crefname{lemma}{Lemma}{Lemmas}
\crefname{proposition}{Proposition}{Propositions}
\crefname{corollary}{Corollary}{Corollaries}
\crefname{definition}{Definition}{Definitions}
\crefname{remark}{Remark}{Remarks}
\crefname{example}{Example}{Examples}

\numberwithin{equation}{section}

\newcommand{\revision}[1]{{#1}}

\newcommand{\p}{\mathbb{P}}

\newcommand{\N}{\mathbb{N}}

\newcommand{\C}{\mathbb{C}}

\newcommand{\Gr}{\mathop{\rm Gr}\nolimits}

\newcommand{\T}{\mathbb{T}}

\newcommand{\Sing}{\mathop{\rm Sing}\nolimits}
\newcommand{\ot}{\otimes}

\title[Characteristic polynomials and eigenvalues of tensors]{Characteristic polynomials\\
and eigenvalues of tensors}

\author[F. Galuppi]{Francesco Galuppi}
\address[F. Galuppi]{Faculty of Mathematics, Informatics and Mechanics, University of Warsaw, Banacha 2, 02-097 Warsaw, Poland (ORCID 0000-0001-5630-5389)}
\email{galuppi@mimuw.edu.pl}
\thanks{Galuppi acknowledges support from the National Science Center, Poland, under the projects ``Complex contact manifolds and geometry of secants'', 2017/26/E/ST1/00231, and ``Tensor rank and its applications to signature tensors of paths'', 2023/51/D/ST1/02363.}

\author[F. Gesmundo]{Fulvio Gesmundo}
\address[F. Gesmundo]{Institut de Mathématiques de Toulouse; UMR5219 -- Université de Toulouse; CNRS -- UPS, F-31062 Toulouse Cedex 9, France (ORCID 0000-0001-6402-021X)}
\email{fgesmund@math.univ-toulouse.fr}

\author[E. T. Turatti]{Ettore Teixeira Turatti}
\address[E. Teixeira Turatti]{Department of Mathematics and Statistics, UiT The Arctic University of Norway, Forskningsparken 1 B 485, Troms\o, Norway (ORCID 0000-0003-4953-3994)}
\email{ettore.t.turatti@uit.no}
\thanks{Turatti acknowledges support from the Troms{\o} Research Foundation grant agreement 17matteCR and from the National Science Center, Poland, under the project ``Tensor rank and its applications to signature tensors of paths'', 2023/51/D/ST1/02363.}

\author[L. Venturello]{Lorenzo Venturello}
\address[L. Venturello]{Dipartimento di Ingegneria dell'Informazione e Scienze Matematiche, Universit\`a di Siena, Italy (ORCID 0000-0002-6797-5270)}
\email{lorenzo.venturello@unisi.it}
\thanks{Venturello is partially supported by the PRIN project ``Algebraic and geometric aspects of Lie theory" (2022S8SSW2).}


\begin{document}

\begin{abstract}

We lay the geometric foundations for the study of the characteristic polynomial of tensors. For symmetric tensors of order $d \geq 3$ and dimension $2$, and symmetric tensors of order $3$ and dimension $3$, we prove that only finitely many tensors share any given characteristic polynomial, unlike the case of symmetric matrices and the case of non-symmetric tensors.  We propose precise conjectures for the dimension of the variety of tensors sharing the same characteristic polynomial, in the symmetric and non-symmetric settings.
\end{abstract}

\maketitle

\section{Introduction}\label{sec: intro}

The classical spectral theory of matrices plays a major role in many areas of pure and applied mathematics. In recent years, several generalizations to the setting of tensors have been proposed, all of which specialize to the same notions for tensors of order two. In this work we are interested in a generalization introduced in \cite{Lim} and \cite{Qi05} which allows one to extend the definitions of characteristic polynomial and eigenvalues to the tensor setting. These notions have applications in spectral hypergraph theory, where spectral properties of the adjacency tensor of a hypergraph carry information on the robustness and other combinatorial features of the hypergraph \cite{OuyQiYua:FirstUnicyclicHypergraphs,YueZhangLuQi:AdjacencySignlessLaplacian,GMV23}. 
\begin{definition}
Let $d$ and $n$ be positive integers with $d \geq 2$. Let $\{v_0 \vvirg v_n\}$ be a basis of $\bbC^{n+1}$  and let
\[
T = \sum_{i_1 \vvirg i_d=0} ^nT_{i_1 \vvirg i_d} v_{i_1} \ootimes v_{i_d}\in (\bbC^{n+1})^{\otimes d}.
\]
A non-zero vector $v = x_0v_0+\dots+x_nv_n \in \bbC^{n+1}$ is an \emph{eigenvector} of $T$ with \emph{eigenvalue} $\lambda$ if 
\[
\sum_{i_2 \vvirg i_d=0}^n T_{i_1 \vvirg i_d} x_{i_2} \cdots x_{i_d} = \lambda x_{i_1}^{d-1} \text{ for every $i_1 \in\{0 \vvirg n\}$}.
\]
\end{definition}
When $d = 2$ the definitions of eigenvalue and eigenvector coincide with the classical definitions in the matrix setting. An important feature of this definition is that the notion of eigenvalue is projective: $v$ is an eigenvector of $T$ with eigenvalue $\lambda$ if and only if every non-zero multiple of $v$ is an eigenvector of $T$ with eigenvalue $\lambda$. We point out that there are other definitions of tensor eigenvalues that are not projective; see, e.g., \cite[Section 2.2.1]{QZ} and \cite{CartSturm}. The \emph{characteristic polynomial} of $T$ is the unique monic univariate polynomial $\phi_T(\lambda)$ whose roots are the eigenvalues of $T$.

A natural problem consists in determining a tensor with prescribed set of eigenvalues, or, more precisely, prescribed characteristic polynomial. An even harder task is to determine \emph{all} tensors with a given characteristic polynomial. This reconstruction problem was studied for instance in \cite{Zay:InvEigenvalueProblem,LDZ} and discussed in broader generality in \cite[Section 3.2]{QiChenChen2018}; in a different setting, a similar problem is discussed in \cite[Section 5.2]{facerecogn}. The problem has a classical solution for matrices: given a monic polynomial $\phi(\lambda)$ with distinct roots, all matrices $T$ such that $\phi_T(\lambda) = \phi(\lambda)$ are conjugate to each other. Moreover, it is easy to exhibit one of them, either by simply considering the diagonal matrix with the roots of $\phi(\lambda)$ as \revision{diagonal} entries, or by considering the \emph{companion matrix} associated to $\phi(\lambda)$. If $\phi(\lambda)$ does not have distinct roots, characterizing all matrices such that $\phi_T(\lambda) = \phi(\lambda)$ involves the Jordan canonical form classification of matrices: this is slightly more technical, but still classical and it only requires basic linear algebra.

For tensors of higher order, the problem is more challenging. First, the number of eigenvalues is larger than the dimension of the tensor factors. Indeed, as observed in \cite[Corollary 1]{Qi05}, as well as in \cite[Theorem 15]{FO} and \cite[Theorem 2.2]{ASS}, if $T \in (\bbC^{n+1})^{\otimes d}$ is generic, then $T$ has $D(n,d) = (n+1)(d-1)^{n}$ distinct eigenvalues: in other words the characteristic polynomial $\phi_T(\lambda)$ is a monic univariate polynomial of degree $D(n,d)$. Further, while any monic univariate polynomial of degree $n+1$ is the characteristic polynomial of some matrix of size $n+1$, for higher order tensors a dimension count shows that the characteristic polynomials form a proper subset of the space of monic univariate polynomials of degree $D(n,d)$, with a small number of exceptions classified in \cite[Theorem 1.1]{YH}. The geometry of this subset is largely mysterious: in general, not even the dimension is known, see \Cref{conj: part sym fibers} below. This piece of information is of major importance for most numerical reconstruction methods, such as gradient and homotopy techniques \cite{BCSS:ComplexityComputation,BurgCu:Condition,timme2021numerical,BGMV:AlgebraicCompressedSensing}. 

In this work, we provide foundational results concerning the locus of tensors with a given (generic) characteristic polynomial. We highlight connections with the theory of resultants, classical intersection theory and projective geometry. We disprove a conjecture proposed in \cite[Section 6.2]{YH}: in particular, we give a lower bound on the dimension of the locus of tensors sharing the same characteristic polynomial as a given generic tensor, see \Cref{lemma: generic orbit in generic fiber}. We propose two conjectures, predicting the exact value of this dimension in general and in the setting of symmetric tensors. For symmetric tensors, we prove the conjecture when $n=1$ and when $(n,d) = (2,3)$: in these cases, we show that for any given tensor $T$, only finitely many tensors $T'$ satisfy $\phi_T (\lambda) = \phi_{T'}(\lambda)$, see \Cref{thm: main}. We make some preliminary investigation for tensors of higher order and in higher dimension, providing evidence for our conjectures. 

In the remainder of the introduction we provide some basic observations in order to state the results more precisely. 

As noted in \cite[Section 5.2]{multiplicity}, the definition of eigenvalue and eigenvector depends only on the \emph{partially symmetric component} of a tensor $T$. More precisely, let $\bbC^{n+1} \otimes S^{d-1} \bbC^{n+1} \subseteq (\bbC^{n+1})^{\otimes d}$ denote the subspace of partially symmetric tensors; this is the subspace invariant under the action of the symmetric group $\frakS_{d-1}$ which permutes the tensor factors $2 \vvirg d$. Then $T$ and its $\frakS_{d-1}$-equivariant projection $T'$ onto $\bbC^{n+1} \otimes S^{d-1} \bbC^{n+1}$ have the same eigenvalues and eigenvectors. For this reason, in the rest of the paper, we restrict our study to the subspace of partially symmetric tensors.

It is natural to define a \emph{characteristic polynomial map}
\begin{equation}\label{eq:charpoly}
\begin{tabular}{cccc}
$\Phi :$ &$\bbC^{n+1} \otimes S^{d-1} \bbC^{n+1}$ &$\to$ &$\bbC[\lambda]^{\mon}_{D(n,d)}$ \\
&$T$ &$\mapsto$ &$\phi_T(\lambda)$.
\end{tabular}    
\end{equation}
Here $\bbC[\lambda]^{\mon}_{D(n,d)}$ is the affine space of monic \revision{univariate} polynomials of degree $D(n,d)$ in the variable $\lambda$. 

It turns out that the characteristic polynomial $\phi_T(\lambda)$ is a univariate restriction of the \emph{resultant polynomial}, an invariant defined on the entire space $\bbC^{n+1} \otimes S^{d-1} \bbC^{n+1}$ \revision{which generalizes the determinant of a matrix}; see, e.g., \Cref{lemma: characteristic polynomial} below. \revision{More precisely, there is a specific tensor $\bfu \in \bbC^{n+1} \otimes S^{d-1}\bbC^{n+1}$ such that $\res(\bfu) \neq 0$ and $\phi_T(\lambda) = \res(T - \lambda \bfu)$; this is analogous to the case $d=2$, when $\bfu$ is the identity matrix and the characteristic polynomial is an affine restriction of the determinant polynomial. In particular the leading term of $\phi_T(\lambda)$ coincides with the value of the resultant polynomial at $\bfu$: this is nonzero, hence up to a global rescaling of $\res$, the characteristic polynomial $\phi_T(\lambda)$ may be assumed to be monic. As a consequence,} the map $\Phi$ is algebraic, in the sense that the \revision{coefficients of the} univariate polynomial $\phi_T(\lambda)$ depend polynomially on the coefficients of $T$; in particular, the map $\Phi$ is a morphism of quasi-projective varieties.

In \cite[Section 6.2]{YH}, it was conjectured that the generic fiber of the characteristic polynomial map $\Phi$ is finite. This is false, as we will observe in \Cref{sec:symmetries}. There is a 
subgroup $H \subseteq \GL(\bbC^{n+1}) \times \GL(\bbC^{n+1})$ acting on the space of tensors $\bbC^{n+1} \otimes S^{d-1} \bbC^{n+1}$ and preserving the characteristic polynomial. Intuitively, this is the natural generalization of the group $\GL(\bbC^{n+1})$ acting on a space of matrices by conjugation, which preserves the eigenvalues.  The group $H$ is isomorphic to the semidirect product $(\bbC^\times)^{n+1} \rtimes \frakS_{n+1}$. We expect the generic fibers of the map $\Phi$ to be orbit-closures for the group $H$, therefore we propose the following conjecture.
\begin{conjecture}\label{conj: part sym fibers}
Let $n \geq 2$,  $d \geq 3$ and $(n,d) \notin\{ (2,3), (3,3),(2,4)\}$. Then the generic fiber of $\Phi$ is equidimensional of dimension $n$.
\end{conjecture}
The cases $n=1$, $d = 2$, and $(n,d) \in\{ (2,3),(3,3),(2,4)\}$ are exactly those where the map $\Phi$ is dominant, \cite[Theorem 1.1]{YH}. In these cases the fiber is equidimensional, and its dimension is given by the difference between the dimension of the domain and the one of the codomain. We record this result, in a slightly more general form, in \Cref{corol: generic fiber dominant cases}. Moreover, \Cref{lemma: generic orbit in generic fiber} shows that the value $n$ in \Cref{conj: part sym fibers} is a lower bound for the dimension of the generic fiber.

However, in applied mathematics, finiteness of the fibers of a map is a desirable recoverability property; see for instance the notion of \emph{algebraic identifiability} in algebraic statistics, \cite[Definition 16.1.1]{sullivant}. 
When it holds, one can apply strong results from geometry to deduce important consequences in the behaviour of numerical methods, as in \cite{BGMV:AlgebraicCompressedSensing}. In order to restrict to a framework where the map $\Phi$ is algebraically identifiable, we notice that in the symmetric setting the analog of the group $H$ is finite. Hence, we consider the restriction $\Phi_{\sym}$ of the map $\Phi$ to the subspace of symmetric tensors $S^d \bbC^{n+1} \subseteq \bbC^{n+1} \otimes S^{d-1} \bbC^{n+1}$. 
\begin{conjecture}\label{conj:fibre symmetric}
If $n \geq 1$ and $d \geq 3$, then all fibers of $\Phi_{\sym}$ are finite.
\end{conjecture}

In \Cref{sec: forme binarie} and \Cref{sec: cubiche piane} we prove \Cref{conj:fibre symmetric} in the case $n=1$ and in the case $(n,d) = (2,3)$. This is our main contribution:
\begin{theorem}\label{thm: main}
Let $d\ge 3$. If $n =1$ or $(n,d) = (2,3)$, then all fibers of $\Phi_{\sym}$ are finite. 
\end{theorem}
The proof of \Cref{thm: main} relies on a strong property of projective morphisms, see \Cref{prop: regular implies finite}: in order to prove \Cref{conj:fibre symmetric}, it suffices to prove that the only tensor satisfying $\phi_T(\lambda) = \lambda^{D(n,d)}$ is $T = 0$. To prove this in the case $n=1$, we use a structural result on the discriminant hypersurface in the case of binary forms, which essentially dates back to \cite{Hilb:UberSingDisc}. In the case $(n,d) = (2,3)$, we rely on the classification of plane cubics, see e.g. \cite{KoganMaza:ComputationCanonicalFormsTernaryCubics}. 

\subsection*{Outlook and future work} A major obstacle in the study of \Cref{conj: part sym fibers} and \Cref{conj:fibre symmetric} is the difficulty in computing symbolic expressions for the resultant and the discriminant polynomials. In particular, a key element of the proof of \Cref{thm: main} is that in the cases $n=1$ and $(n,d) = (2,3)$, we are able to determine the highest possible intersection multiplicity between a line and the discriminant hypersurface. We expect the next easiest cases to be those classified in \cite[Proposition 13.1.6]{GKZ:DiscResMultDet}, where the resultant polynomial has a determinantal formula.

Regarding the structure of the fiber of $\Phi$ and $\Phi_{\sym}$, it would be interesting to compute the number of irreducible components. From the discussion in \Cref{sec:symmetries}, one would expect that the generic fiber of $\Phi_\sym$ consists of $d^n (n+1)!$ points; in fact, this is a lower bound on the cardinality of the generic fiber. We verify numerically that it is the correct cardinality for $(n,d) \in\{(1,5),(1,6)\}$. However, both for $(n,d)=(1,3)$ and for $(n,d)=(1,4)$, we obtain fibers of cardinality $24$, instead of the expected $6$ and $8$. The value $24$ in the case $(1,3)$ is explained in \Cref{remark: binary cubic surjective}. We could not explain the value $24$ for the case $(1,4)$, but we expect this to reflect the fact that the generic binary form of degree $4$ does not have trivial isotropy group, unlike binary forms of higher degree $d \geq 5$.

We propose the study of the variety of characteristic polynomials, that is the (closure of the) images of the maps $\Phi$ and $\Phi_\sym$, in the spirit of \cite[Section 2]{reconstructionsing}, which studies the variety of singular values. Answering \Cref{conj: part sym fibers} and \Cref{conj:fibre symmetric} would provide the dimension of these varieties. \revision{\Cref{thm: main} guarantees that when $n=1$ the image of $\Phi_{\sym}$ is closed. Computer investigation provides the following forward-looking results.}
\begin{itemize}[leftmargin=*]
\item If $(n,d)= (1,3)$, then $\Phi_\sym$ is surjective. This is the only case where $\Phi_\sym$ is dominant for $d\geq 3$. 
\item If $(n,d) = (1,4)$, then the image of $\Phi_\sym$ is a hypersurface of degree $30$ in $\bbC^{6}$. The degree was first obtained numerically, then verified symbolically using standard interpolation methods.
\item If $(n,d) = (1,5)$, then the image of $\Phi_\sym$ is a variety of codimension $2$ in $\bbC^8$. Numerical experiments performed using classical monodromy in \texttt{HomotopyContinuation.jl} \cite{BreTim:HomotopyContinuation} show that its degree is at least 2010.
\end{itemize}

It would be fascinating to link the characteristic polynomial of a tensor to its rank. For matrices, the rank is the number of non-zero eigenvalues, but tensor rank is notoriously more mysterious than matrix rank.

We point out that in some applications one restricts the study of eigenvalues and eigenvectors to the real numbers. In this work, we focused on the complex setting, which is more natural from a projective geometry perspective. However, it would be interesting to study the distribution of real eigenvalues of real (symmetric) tensors. In particular, unlike the case of matrices, eigenvalues of symmetric tensors are not necessarily all real. This problem has been preliminarily studied in \cite{ASS}, but its geometry is far from understood. 

\subsection*{Acknowledgments} The semester program Applications to Tensors and Secants (AGATES), held in Warsaw in 2022, was a great place to work and to discuss. We thank Jarek Buczy\'nski and Maciej Ga{\l}{\k{a}}zka for helpful suggestions on the projective geometry of the characteristic polynomial map.

This work is partially supported by the Thematic Research Programme ``Tensors: geometry, complexity and quantum entanglement'', University of Warsaw, Excellence Initiative -- Research University and the Simons Foundation Award No. 663281 granted to the Institute of Mathematics of the Polish Academy of Sciences for the years 2021--2023. 


\section{Preliminaries}\label{sec:preliminaries}
In this section we provide the formal definitions of eigenvalues and eigenvectors of tensors. In \Cref{corol: generic fiber dominant cases}, we determine the dimension of the fiber of the map $\Phi$ in the dominant cases classified in \cite[Theorem 1.1]{YH}. We conclude the section by reducing our study of symmetric tensors to the projective setting: we show that \Cref{conj:fibre symmetric} holds if and only if a projectivized version of the characteristic polynomial map is well-defined. This is the result that will ultimately yield the proof of \Cref{thm: main}.

\begin{definition}\label{def: eigenvalues}
Let $V$ and  $W$ be $\C$-vector spaces. Let $d \geq 2$ be an integer and let $\bft\in V \otimes {W^*}^{\otimes (d-1)}$. For a tensor $T \in V \otimes {W^*}^{\otimes(d-1)}$, a \emph{$\bft$-eigenvector} of $T$ with \emph{$\bft$-eigenvalue} $\lambda$ is a vector $w \in W\setminus\{0\}$ with the property that 
\[
T  ( w^{\otimes (d-1)})  = \lambda \bft  ( w^{\otimes (d-1)}) 
\]
as elements of $V$. The pair $(\lambda, w)$ is a \emph{$\bft$-eigenpair} for $T$.
\end{definition}

Without loss of generality, one may restrict to partially symmetric tensors. More precisely, let $S^{d-1}W^*$ be the symmetric subspace of ${W^*}^{\otimes (d-1)}$, that is the space which is invariant under the action of the symmetric group $\frakS_{d-1}$ permuting the factors of ${W^*}^{\otimes (d-1)}$; let $\pi_S$ be the canonical projection of $V \otimes {W^*}^{\otimes (d-1)}$ onto $V \otimes S^{d-1}W^*$, that is 
\[\begin{tabular}{cccc}
$\pi_S :$ & $V \otimes {W^*}^{\otimes (d-1)}$ &$\to$ & $V \otimes S^{d-1}W^*$ \\
&$v \otimes w_1 \ootimes w_{d-1}$ &$\mapsto$ & $ v \otimes  \textfrac{1}{(d-1)!}\textstyle \sum_{\sigma \in \frakS_{d-1}}  w_{\sigma(1)} \ootimes w_{\sigma(d-1)}$.
\end{tabular}\]
Then we have the following immediate result:
\begin{lemma}\label{lemma: partially sym compt}
Let $T , \bft \in V \otimes {W^*}^{\otimes (d-1)}$. The $\bft$-eigenvectors and eigenvalues of $T$ only depend on $\pi_S(T)$ and $\pi_S(\bft)$. 
\end{lemma}
\begin{proof}
    Write $T = \pi_S(T) + T'$ and $\bft = \pi_S(\bft) + \bft'$ for certain tensors $T',\bft'$, uniquely determined by the projection. The result follows by linearity, since $T'(w^{\otimes (d-1)}) = \bft'(w^{\otimes (d-1)})  = 0$. 
\end{proof}
Following \Cref{lemma: partially sym compt}, we will assume that $\bft \in V \otimes S^{d-1} W^*$ and restrict our study to tensors $T \in V \otimes S^{d-1} W^*$. It is a standard fact that elements of $S^{d-1} W^*$ can be regarded as homogeneous polynomials of degree $d-1$ on $W$, hence elements of $V \otimes S^{d-1} W^*$ can be identified with polynomial maps $W \to V$ whose components (in any given choice of coordinates) have degree $d-1$. If $T \in V \otimes S^{d-1} W^*$ and $w \in W$, write $T(w) = T ( w^{\otimes (d-1)}) \in V$ for the image of $T$ via the contraction with $w$ on all the factors $W^*$.

In practice, one is interested in $\bft$-eigenpairs for a particular choice of $\bft$, that is the \emph{unit tensor}. If $\dim V = \dim W = n+1$, let $\{v_0 \vvirg v_n\}$ and  $\{w_0 \vvirg w_n\}$ be bases of $V$ and $W$ respectively, and let $\{x_0 \vvirg x_n\}$ be the basis of $W^*$ dual to the chosen basis of $W$; explicitly $x_j$ is the linear map sending a vector $w=\mu_0w_0+\dots+\mu_nw_n \in W$ to the coefficient 
$\mu_j$. The \emph{unit tensor of rank $n+1$} is $$\bfu_{n+1} = \sum_{i=0}^n v_i \otimes x_i^{d-1}.$$
In the coordinate representation of $\bfu_{n+1}$ as a $d$-dimensional array of numbers, the entries of $\bfu_{n+1}$ are $1$ on the diagonal and $0$ elsewhere.  The eigenpairs defined in \cite[Section 2.1]{QZ} and discussed in \Cref{sec: intro} are exactly the $\bfu_{n+1}$-eigenpairs; the more general setting is considered for instance in \cite[Example 3.6]{MulNanSei:MultilinearHyperquiver}, in connection to the representation theory of hyperquivers.

\begin{remark}\label{rmk: for matrices we basically choose t equals u}
If $d=2$, then $\bfu_{n+1}$ is, in coordinates, the identity matrix of size $n+1$, hence the $\bfu_{n+1}$-eigenpairs of a tensor $T$ of order two are the classical eigenpairs of the matrix $T$. More generally, if $\bft$ is a tensor of order two whose representing matrix is non-singular, then the $\bft$-eigenpairs of a tensor $T$ are the classical eigenpairs of the matrix $\bft^{-1}T$. These eigenvalues are central in the study of matrix pencils  \cite[Chapter XII]{Gant:TheoryOfMatrices} and generalized eigenvalue problems \cite[Section 7.7]{GolVL:MatrixComputations}. Therefore \Cref{def: eigenvalues} can be viewed as a generalization to the tensor setting of the spectral theory of matrix pencils.
\end{remark}

\subsection{Characteristic polynomial of tensors}
In order to define the characteristic polynomial, we suppose that $$\dim V = \dim W = n+1$$ and we will keep this assumption for the rest of the paper. Let $T \in V \otimes S^{d-1}W^*$. The image of the flattening map $T : V^* \to S^{d-1} W^*$ defines a linear space of polynomials of degree $d-1$. For a generic choice of $T$, this linear space of polynomials does not have a common zero in $\bbP W$; the subset of tensors for which this linear space has a common zero in $\bbP W$ is an irreducible hypersurface of degree $D(n,d) = (n+1)(d-1)^n$ in $\bbP (V \otimes S^{d-1}W^*)$, called \emph{resultant} \cite[Section 13.1]{GKZ:DiscResMultDet}. Denote by $\calRes \subseteq \bbP (V \otimes S^{d-1}W^*)$ this hypersurface and by $\res \in \bbC[V \otimes S^{d-1}W^*]_{D(n,d)}$ its defining polynomial. We record the following result, which is an extension of \cite[Theorem 2.12(a)]{QZ}. \revision{In the following, we will assume $\res(\bft) \neq 0$.}
\begin{lemma}\label{lemma: characteristic polynomial}
 Let $T,\bft \in V \otimes S^{d-1}W^*$ and let $\lambda_0\in\C$. The following are equivalent:
 \begin{itemize}
     \item $\lambda_0$ is a $\bft$-eigenvalue of $T$;
     \item $\lambda_0$ is a root of the univariate polynomial $\res( T - \lambda \bft) \in \bbC[\lambda]$.
 \end{itemize}
In particular, if $T$ and $\bft$ are general then $T$ has $D(n,d)$ $\bft$-eigenpairs.
\end{lemma}
\begin{proof}
The scalar $\lambda_0$ is a $\bft$-eigenvalue of $T$ if and only if there exists a vector $w \in W$ with the property that $T ( w ) = \lambda_0 \bft(w) \in V$. This is equivalent to the fact that the polynomial system corresponding to the tensor $T- \lambda_0 \bft$ has a non-trivial solution $w$, which by definition is equivalent to the condition $\res(T- \lambda_0 \bft) = 0$. 
\end{proof}

\begin{remark}
The number of $\bft$-eigenpairs of a generic tensor $T$ can be recovered as the degree of the $n$-th Chern class of the vector bundle $\calE = V \otimes \calO_{\bbP W} (d-1)$ on $\bbP W$. Indeed, the tensors $T$ and $\bft$ can be regarded as two global sections of $\calE$; since $\calE$ has rank $n+1$, by definition its $n$-th Chern class detects when $T$ and $\bft$ are linearly dependent as sections of $\calE$. If $h_W$ is the hyperplane class on $\bbP W$, then $c_1(\calO_{\bbP W}(d-1)) = (d-1) h_W$ and $c(\calE) = (1 + (d-1)h_W)^{n+1}$; the component of degree $n$ is indeed $c_n(\calE) = (n+1) (d-1)^{n} h_W^n$ which has degree $D(n,d)$.
\end{remark}

\begin{remark}\label{remark: clearly}
\revision{Since we assume $\res(\bft)\neq 0$, every $\bft$-eigenvector of $T$ has a unique $\bft$-eigenvalue. Indeed, assume that a $\bft$-eigenvector $w$ of a tensor $T$ has two $\bft$-eigenvalues $\lambda$ and $\mu$. Then $T(w^{\ot (d-1)})=\lambda\bft(w^{\ot (d-1)})=\mu \bft(w^{\ot (d-1)})$. But our hypothesis implies that $\bft(w^{\ot (d-1)})\neq 0$, so $\lambda=\mu$. }
\end{remark}

Define $\phi^\bft_T(\lambda)  = \res(T- \lambda \bft)$ to be the \emph{$\bft$-characteristic polynomial} of $T$. \Cref{lemma: characteristic polynomial} guarantees that the roots of $\phi^\bft_T(\lambda)$ are the $\bft$-eigenvalues of $T$. Since $\res(\bft) \neq 0$, we have that $\phi^\bft_T(\lambda)$ has degree $D(n,d)$ \revision{and its leading coefficient is $\res(\bft)$, so it may be assumed to be $1$ after rescaling $\res$ by a global constant}. Therefore, every element $\bft$ with $\res(\bft) \neq 0$ gives rise to a \emph{characteristic polynomial map}
\[
\begin{tabular}{ccccc}
$\Phi_\bft :$&$ V \otimes S^{d-1} W^*$ &$\to$&$ \bbC[\lambda]_{D(n,d)}^{\mon} \cong \bbA^{D(n,d)}$ \\
&$T$ &$\mapsto$&$\res(T -\lambda \bft)$.
\end{tabular}
\]

In \cite[Theorem 1.1]{YH}, the cases where the map $\Phi_{\bfu_{n+1}}$ is dominant are classified: this is the case for $n=1$, when $d=2$, and when $(n,d) = (2,3),(3,3),(2,4)$. In these cases, it is easy to prove  equidimensionality of the generic fiber, using the following standard fact. 
\begin{lemma}\label{lemma: krull dimension}
    Let $f : \bbA^N \to \bbA^M$ be a dominant map of affine varieties. Then the generic fiber is  equidimensional of dimension $N-M$.
\end{lemma}
\begin{proof}
    Let $x \in \bbA^N$ be generic and let $F_x = f^{-1}(f(x))$ be the fiber over $f(x)$. Since $f$ is dominant, $\dim F_x = N-M$ and it possibly has components of lower dimension. On the other hand, the fiber $F_x$ is set-theoretically defined as $F_x = \{ x' \in \bbA^N \mid f(x') = f(x)\}$. In coordinates, the condition $f(x) = f(x')$ consists of $M$ equations on $\bbA^N$. An easy consequence of Krull's Hauptidealsatz \cite[Corollary 11.16]{AtMac:CommutativeAlgebra} is that every component of $F_x$ has codimension at most $M$, as desired.
\end{proof}

\Cref{lemma: krull dimension} guarantees that the fiber of $\Phi_{\bfu_{n+1}}$ is equidimensional in the cases where the map is dominant. An easy semicontinuity argument guarantees that analogous results hold for every generic enough $\bft \in V \otimes S^{d-1}W^*$. We record this explicitly in the next result.
\begin{corollary}\label{corol: generic fiber dominant cases}
Let $n=1$, or $d = 2$, or $(n,d)\in\{ (2,3),(3,3),(2,4)\}$. Let $\bft \in V \otimes S^{d-1} W^*$ be generic or $\bft = \bfu_{n+1}$. Then $\Phi_\bft$ is dominant and its generic fiber is equidimensional of dimension $(n+1) \revision{\binom{n+d-1}{d-1}} - D(n,d)$. 
\end{corollary}
\begin{proof}
The fact that $\Phi_\bft$ is dominant when $\bft = \bfu_{n+1}$ follows from \cite[Theorem 1.1]{YH}. For general $\bft$, the same result is a consequence of a standard semicontinuity argument. The dimension of the generic fiber is the difference between the dimension of the domain and the one of the image. Equidimensionality follows from \Cref{lemma: krull dimension}. 
\end{proof}

In order to prove \Cref{thm: main}, we will study a projectivized version of the map $\Phi_\bft$. Indeed, in the monomial basis of the space of monic polynomials, the $i$-th component of $\Phi_\bft$ is a homogeneous polynomial of degree $i$ in the coordinates of $T$. Therefore $\Phi_\bft$ defines a rational map to a \emph{weighted projective space} $\bbP(\bfrho_{D(n,d)})$ of dimension $D(n,d)-1$ and weights $\bfrho_{D(n,d)} = (1,2\vvirg D(n,d))$:
\begin{center}
    \begin{tabular}{cccc}
$\Phi_\bft^{\bbP} :$& $\bbP ( V \otimes S^{d-1} W^* )$ &$\dashto$ & $\bbP(\bfrho_{D(n,d)})$ \\
 &$[T]$ &$\mapsto$ & $(c_1(T) \vvirg c_{D(n,d)}(T)) $,
\end{tabular}
\end{center}
where $\phi^\bft_T(\lambda) = \lambda^{D(n,d)} + c_1(T) \lambda^{D(n,d)-1} + \cdots + c_{D(n,d)}(T)$. We will use the same notation for the affine and the projective version of the characteristic polynomial map and refer to them as affine or projective when confusion might arise. Consider the restriction of the characteristic polynomial map to the subspace of symmetric tensors as well. Consider an identification $V \cong W^*$ and let $S^d W^* \subseteq W^* \otimes S^{d-1} W^*$ be the subspace of symmetric tensors, naturally identified with the space of homogeneous polynomials of degree $d$ on $W$. The resultant polynomial $\res$ restricts to the \emph{discriminant polynomial} $\disc \in \bbC[S^d W^*]_{D(n,d)}$, which defines the hypersurface $\calDisc$ consisting of polynomials of degree $d$ whose zero set is singular. Given $\bft \in S^d W^*$ with $\disc(\bft) \neq 0$, let 
\begin{center}
    \begin{tabular}{cccc}
$\Phi_{\bft,\sym}  :$&$ S^{d} W^* $&$\to $&$\bbC[\lambda]_{D(n,d)}^{\mon} \cong \bbA^{D(n,d)} $\\
&$T$ &$\mapsto$&$ \disc(T -\lambda \bft)$
\end{tabular}
\end{center}
be the restriction of the characteristic polynomial map to the symmetric subspace. Let $\Phi_{\bft,\sym}^{\bbP} : \bbP S^d W^* \dashto \bbP ( \bfrho_{D(n,d)})$ be the rational map induced on projective space.

We conclude this section by proving \Cref{prop: regular implies finite}, stating a general geometric fact which will be the key to prove \Cref{thm: main}.

\begin{proposition}\label{prop: regular implies finite}
Let $f : \bbP^N \to Y$ be a surjective morphism of projective varieties. Then either $f$ is constant, or $\dim Y = N$ and all fibers of $f$ are finite.
\end{proposition}
\begin{proof}
Suppose $f$ is not constant and assume by contradiction that there exists $y \in Y$ with a positive dimensional fiber. Let $F$ be a positive dimensional component of $f^{-1}(y)$. Since $f$ is not constant, $Y$ has dimension at least $1$, and we can consider an irreducible divisor $D\subset Y$ such that $y \notin D$ and with the property that the generic element $z \in D$ satisfies $\dim f^{-1}(z) = N - \dim Y$. Therefore $f^{-1}(D)$ has at least one irreducible component $Z$ which is a hypersurface in $\bbP^N$. Since $F$ has positive dimension, we deduce $F \cap Z \neq \emptyset$. On the other hand 
    \[
F \cap Z \subseteq f^{-1} ( y) \cap f^{-1} (D) 
   = f^{-1} ( \{ y \} \cap D) = \emptyset
   \]
and this is a contradiction. 
\end{proof}

\begin{remark}\label{rmk: it is enough to prove that phi is defined everywhere}
By \Cref{prop: regular implies finite}, every non-constant projective morphism whose domain is a projective space has indeed finite fibers. As a consequence, we deduce that if the characteristic polynomial map $\Phi_{\bft,\sym}^{\bbP} : \bbP (S^{d} W^*) \dashto \bbP(\bfrho_{D(n,d)})$ is defined everywhere, then all its fibers are finite. We will prove that this is the case when $n =1 $ and when $(n,d) = (2,3)$. \Cref{thm: main} is the special case $\bft = \bfu_{n+1}$.

The indeterminacy locus of the maps $\Phi_{\bft}^{\bbP}$ and $\Phi_{\bft,\sym}^{\bbP}$ consists of tensors $T$ such that $\phi_T^\bft(\lambda) = \lambda^{D(n,d)}$, namely, tensors whose only eigenvalue is $0$. We will prove these tensors do not exist under the assumptions of \Cref{thm: main}. More precisely, we will see that the non-emptiness of the indeterminacy locus is equivalent to the existence of lines in $\bbP S^d W^*$ intersecting the discriminant hypersurface $\calDisc$ at a single point and the proof will be built on certain intersection theoretic properties of such hypersurface.
\end{remark}

\section{Symmetries preserving eigenvalues and eigenvectors}\label{sec:symmetries}
In this section we discuss the action of $\GL(V)\times\GL(W)$ on the space of partially symmetric tensors $V\ot S^{d-1}W^*$. This will give us a lower bound on the dimension of the fibers of $\Phi_\bft$, providing evidence for \Cref{conj: part sym fibers} and answering a question raised in \cite[Section 6.2]{YH}.

Let $(g_V,g_W)\in\GL(V)\times\GL(W)$ and let $T\in V\ot S^{d-1}W^*$. Regard $V\ot S^{d-1}W^*$ as the space of polynomial maps $W\to V$, and define $(g_V,g_W)\cdot T$ to be the polynomial map $W\to V$ given by
\begin{align*}
\revision{w \mapsto g_V T(g_W^{-1}(w))}.
\end{align*}
Notice that this action is not faithful. Indeed, the subgroup $C = \{ (\alpha^{d-1}\id_V , \alpha \id_W) \mid \alpha \in \bbC^{\times}\}$ of $\GL(V) \times \GL(W)$ acts as the identity on $V \otimes S^{d-1} W^*$.

\begin{definition}
Let $\bft\in V\ot S^{d-1}W^*$. The stabilizer
\[
G_\bft=\{(g_V,g_W)\in\GL(V)\times\GL(W)\mid (g_V,g_W)\cdot\bft=\bft\}
\]
is called the \emph{isotropy group} of $\bft$.
\end{definition}

The eigenvalues and the characteristic polynomial of a tensor are invariant under the action of the isotropy group of the tensor $\bft$, as observed in the following result.

\begin{lemma}\label{lemma: same eigenvalues via stabilizer}
Let $\bft,T \in V \otimes S^{d-1}W^*$ and let $g = (g_V,g_W) \in G_\bft$. Then $(\lambda,w)$ is a $\bft$-eigenpair for $T$ if and only if $(\lambda, g_Ww)$ is a $\bft$-eigenpair for $g \cdot T$. In particular, $T$ and $g \cdot  T$ have the same $\bft$-characteristic polynomial.
\end{lemma}
\begin{proof}
Let $(\lambda, w)$ be a $\bft$-eigenpair for $T$. Then
\begin{align*}
(g \cdot T) ( (g_W w)^{\otimes (d-1)}) = &g_VT (g_W^{-1} (g_W w)^{\otimes (d-1)}) = g_V T(w^{\otimes (d-1)}) = \\ & \lambda g_V \bft(w^{\otimes (d-1)}) =  \lambda g_V \bft(g_W ^{-1} g_W w^{\otimes (d-1)}) = \\ 
&\lambda (g\cdot \bft)((g_W w)^{\otimes (d-1)}) = \lambda \bft ( (g_W w)^{\otimes (d-1)} ).
    \end{align*}
This shows that $(\lambda , g_W w)$ is a $\bft$-eigenpair of $g \cdot T$. Since $g \in G_{\bft}$ implies $g^{-1} \in G_\bft$ as well, the other implication follows.
\end{proof}

When $\dim V=\dim W=n+1$ and $\bft$ is the unit tensor $\bfu_{n+1}$, we can explicitly describe $G_{\bfu_{n+1}}$. The literature contains several proofs of this result either for tensors of order three, such as  \cite[Section 3]{dGro:VarsOptAlgIIsoGrps} and \cite[Section 4.1]{BuIk:GCT_and_tensor_rank} or in the setting of symmetric tensors \cite[Section 8.12.1]{Lan:GeometryComplThBook}. These proofs generalize to the partially symmetric setting; we provide an outline for completeness. It is also possible to give an explicit proof, following the methods from \cite[Section 5]{GesIkPa:GCTMatrixPowering}.
 
When $d=2$, the unit tensor is the $(n+1)\times (n+1)$ identity matrix, so $G_{\bfu_{n+1}}$ is a copy of $\GL_{n+1}$ embedded diagonally in $\GL(V) \times \GL(W)$.

For $d\ge 3$, let $\T_{n+1}=(\C^\times)^{n+1}$ be the algebraic torus of rank $n+1$. Consider the embedding of $\T_{n+1}$ in $\GL(W)$ with weights $(1,\dots,1)$ and its embedding in $\GL(V)$ with weights $(d-1,\dots,d-1)$ with respect to the bases chosen to define the unit tensor. Explicitly, if $(a_0,\dots,a_n)\in\T_{n+1}$, $v_i$ is a basis element of $V$ and $w_j$ is a basis element of  $W$, then
\[
a\cdot v_i = a_i^{d-1} v_i , \qquad a\cdot w_j = a_j \cdot w_j.
\]
Under this embedding, $\T_{n+1}$ acts on $V \otimes S^{d-1}W^*$. Explicitly, if $T = v_i \otimes x_{j_1} \cdots x_{j_{d-1}}$ is a basis element of $V \otimes S^{d-1} W^*$, then 
\[
a \cdot T = (a_i^{d-1} v_i) \otimes  (a_{j_1}^{-1} x_{j_1}) \cdots (a_{j_{d-1}}^{-1} x_{j_{d-1}}) = a_i^{d-1} (a_{j_1} \cdots a_{j_{d-1}})^{-1}T.
\]
In \cite[Section 3]{YangYang}, it was shown that this action preserves eigenvalues of tensors in $V \otimes S^{d-1}W^*$. 

Let $\frakS_{n+1}$ be the symmetric group on $n+1$ elements, which naturally acts on $V$ and $W$ by permuting the basis elements: 
\[
\sigma \cdot v_i = v_{\sigma(i)}, \qquad \sigma \cdot w_j = w_{\sigma(j)} .
\]
In \cite[Theorem 3]{permutations}, it was shown that this action preserves eigenvalues of tensors in $V \otimes S^{d-1}W^*$ as well. Moreover, the embedded copy of $\frakS_{n+1}$ normalizes the embedded copy of $\bbT_{n+1}$ in $\GL(V) \times \GL(W)$. Let $H = \bbT_{n+1} \rtimes \frakS_{n+1}$ be the resulting semidirect product. 
\begin{lemma}\label{lemma: u is stabilized only by YY and permutations}
If $d \geq 3$, then $H$ is the isotropy group of the unit tensor $\bfu_{n+1} = \sum_0^n v_i \otimes x_i^{d-1}$. In other words
\[
G_{\bfu_{n+1}} = \bbT_{n+1} \rtimes \frakS_{n+1},
\]
with the action described above.
  \end{lemma}
\begin{proof}
It is clear that $H$ stabilizes $\bfu_{n+1}$, therefore $H \subseteq G_{\bfu_{n+1}}$. The fixed bases define an identification $V \cong W$. Moreover, via the fixed bases, $W$ can be regarded as an $(n+1)$-dimensional algebra with component-wise product $w_i \cdot w_j = \delta_{ij} w_j$.

It is easy to verify that $\bfu_{n+1}$ is the structure tensor of the $(d-1)$-fold product of this algebra structure: as a $(d-1)$-linear map, it maps $d-1$ elements of $W$ to their product. Then the higher order analog of \cite[Theorem 3.1]{dGro:VarsOptAlgIIsoGrps} and \cite[Proposition 4.1]{BuIk:GCT_and_tensor_rank} guarantee that $H$ is the desired isotropy group. 
\end{proof}

The dimension of the isotropy group provides a lower bound on the dimension of the generic fibers of $\Phi_{\bft}$. This answers the question posed in \cite[Section 6.2]{YH} regarding the dimension of the generic fiber of $\Phi_{\bfu_{n+1}}$: such a fiber is never finite.

\begin{proposition}\label{lemma: generic orbit in generic fiber}
Let $d \geq 3$ and let $\bft \in V \otimes S^{d-1} W^*$ be a tensor with $\res(\bft) \neq 0$. Then the generic fiber of $\Phi_\bft$ has dimension at least $\dim G_{\bft} -1$. In particular, the generic fiber of $\Phi_{\bfu_{n+1}}$ has dimension at least $n$.
\end{proposition}
\begin{proof}
Let $C = \{ ( \alpha^{d-1} \id_{V} , \alpha \id_W) \in \GL(V) \times \GL(W)\mid \alpha \in \bbC^\times\}$. Since $C$ acts trivially on $V \otimes S^{d-1} W^*$, we have $C \subseteq G_T$ for every $T\in V \otimes S^{d-1} W^*$. 

First, assume $(n,d)=(1,3)$. In this case, the space $V \otimes S^2 W^*$ has finitely many orbits under the action of $\GL(V) \times \GL(W)$ \cite{Sylv:PrinciplesCalculusForms}. There is a unique orbit whose elements satisfy $\res(\bft) \neq 0$; this is the orbit of $\bfu_2$, which is dense. In particular, in this case we may assume $\bft = \bfu_2$, hence $\dim G_{\bft} = 2$ by \Cref{lemma: u is stabilized only by YY and permutations}. An explicit calculation shows that if $T \in V \otimes S^2 W^*$ is generic, then $G_T \cap G_\bft = C$, hence the $G_\bft$-orbit of $T$ has dimension $1$. By \Cref{lemma: same eigenvalues via stabilizer}, all the elements of this orbit have the same characteristic polynomial, hence $\dim \Phi_\bft^{-1}\Phi_\bft(T) \geq 1$ as desired. 

If $d \geq 4$ or $n \geq 2$, then a generic element $T \in V \otimes S^{d-1} W^*$ satisfies $\dim G_T = \dim C = 1$. This fact is classical and it can be seen as a consequence of \cite{BryReiVRaa:ExistenceLocMaxEntStatesViaGIT}. Explicitly, by a semicontinuity argument, it is enough to exhibit an element $T \in V \otimes S^{d-1} W^*$ satisfying $\dim G_T = 1$. When $d = 3$, several examples of such tensors are given in \cite{ConGesLanVenWan:GeometryStrassenAsyRankConj} and they easily generalize to higher order. For instance the tensor $T = \bfu_{n+1} + (\sum_0^n v_i) \otimes (\sum_0^n x_i)^{d-1}$ satisfies $\dim G_T = 1$. In particular $\dim (G_T \cap G_\bft) = \dim C$. Therefore the $G_\bft$-orbit of $T$ has dimension $\dim G_\bft -1$, and by \Cref{lemma: same eigenvalues via stabilizer} it is contained in the fiber $\Phi_\bft^{-1}\Phi_\bft(T)$. This proves the first part of the statement.

For the special case $\bft=\bfu_{n+1}$, \Cref{lemma: u is stabilized only by YY and permutations} implies that the general fiber of $\Phi_{\bfu_{n+1}}$ has dimension at least $\dim G_{\bfu_{n+1}} -1= \dim\T_{n+1}-1=n$.
\end{proof}

In the symmetric setting, one restricts the action to the subgroup of $\GL(V) \times \GL(W)$ preserving the space of symmetric tensors, that is a diagonal copy of $\GL(W)$. Let $G^\sym_\bft$ be the isotropy group of a symmetric tensor $\bft$ with respect to the action. Note that the subgroup $C^\sym = \{ \alpha \ \id_W \mid \alpha^d = 1\}$ is contained in $G^\sym_\bft$ for every tensor $\bft$; since $C^\sym$ is finite, we have $\dim C^{\sym} = 0$. One has an analog of \Cref{lemma: generic orbit in generic fiber}:

\begin{proposition}
Let $d \geq 3$ and let $\bft \in S^{d} W^*$ be a symmetric tensor with $\disc(\bft) \neq 0$. Then the generic fiber of $\Phi_{\bft, \sym}$ has dimension at least $\dim G^{\sym}_{\bft}$.
\end{proposition}
\begin{proof}
The proof is similar to the one of \Cref{lemma: generic orbit in generic fiber}.
\end{proof}
Under the identification $V \cong W^*$ induced by the choice of basis, we have $\bfu_{n+1} = \sum_0^n x_i^d \in S^d W^*$. The isotropy group of the symmetric tensor $\bfu_{n+1}$ is the finite group $G_{\bfu_{n+1}}^\sym \cong (\bbZ_d)^{n+1} \rtimes \frakS_{n+1} \subseteq \GL(W)$; we refer to \cite[Section 8.12.1]{Lan:GeometryComplThBook} for an explicit proof. Since $G_{\bfu_{n+1}}^\sym$ is finite, we do not obtain a lower bound on the dimension of the generic fibers of $\Phi_{\bfu_{n+1},\sym}$, unlike the case of general tensors. In this setting, we expect the generic fiber of $\Phi_{\bfu_{n+1},\sym}$ to be an orbit for the action of $ G_{\bfu_{n+1}}^\sym$, at least when $n$ and $d$ are sufficiently large; if this is the case, then the generic fiber is finite, as predicted by \Cref{conj:fibre symmetric}, and it has degree $d^n (n+1)!$.

\section{Multiplicity of eigenvalues and higher Hurwitz forms}
In this section we study the indeterminacy locus of the characteristic polynomial maps $\Phi^\p_\bft$ and $\Phi^\p_{\bft,\sym}$. Following \Cref{rmk: it is enough to prove that phi is defined everywhere}, the resolution of these maps is the key strategy to prove \Cref{thm: main}. We will describe the indeterminacy loci in terms of intersection properties of lines with the resultant and the discriminant hypersurfaces. To do so, we recall the definitions of tangent cones and multiplicity, together with some of their elementary properties, and their consequences in the context of characteristic polynomials. In particular, we will show a tensor analog of the fact that matrices with distinct eigenvalues have linearly independent eigenvectors, see \Cref{thm: distinct evals implies distinct evecs new}. 

We point out that \Cref{lemma: characteristic polynomial} gives a natural generalization of the algebraic multiplicity of an eigenvalue: it is the multiplicity of $\lambda$ as a root of the characteristic polynomial. Generalizing the notion of geometric multiplicity is more delicate. Some generalizations have been proposed in \cite[Section 1]{multiplicity} and in \cite{Qi05} in terms, for instance, of the dimension of the eigenscheme of $T$, or of the dimension of its linear span.

\subsection{Tangent cones and multiplicities}

We recall the notion of tangent cone and multiplicity in the case of hypersurfaces, and we refer to \cite[Chapter 2]{Shaf:BasicAlgGeom1} for an extensive discussion.

Let $X \subseteq \bbP^N$ be the hypersurface defined by the homogeneous polynomial $F \in \bbC[x_0\vvirg x_N]$. The \emph{multiplicity} of $X$ at the point $p \in \bbP^N$ is 
\[
\mult_X(p) = \min \left\{ m\in\N\mid \partial_\alpha F (p) \neq 0 \text{ for some multi-index $\alpha$ of length $m$}\right\},
\]
where $\partial_\alpha F = \frac{\partial F} {\partial \bfx^\alpha}$. For example, $\mult_X(p) \geq 1$ if and only if $p \in X$, and equality holds if and only if $p$ is a smooth point of $X$. If $p \in X$ and $m = \mult_X(p)$, then the \emph{tangent cone} to $X$ at $p$ is 
\[
TC_p X = \left\{ [\bfx] \in \bbP^N \mid \sum_{ | \alpha | = m} 
\partial_\alpha F(p) \cdot \bfx^\alpha= 0\right\} \subseteq \bbP^N.
\]
It is a hypersurface of degree $m$ which is a cone with vertex $p$. If $\mult_X(p) = 1$, then $TC_p X$ is a linear space and it coincides with the tangent space to $X$ at $p$.

We record an immediate result:
\begin{lemma}\label{lemma: multiplicity bounded from below by mult new}
    Let $\bft,T \in V \otimes S^{d-1} W^*$ with $\res(\bft) \neq 0$. If $\lambda_0$ is a $\bft$-eigenvalue of $T$ with algebraic multiplicity $m$, then $m \geq \mult_{\calRes}(T -\lambda_0\bft)$.
\end{lemma}
\begin{proof}
\revision{Expand $\phi^\bft_\lambda(T) = \res(T - \lambda \bft)$ as a univariate polynomial in $(\lambda - \lambda_0)$.} Up to scaling, the coefficient of \revision{$(\lambda - \lambda_0)^k$} is the $k$-th directional derivative of $\res$ in the direction $\bft$, evaluated at $T$. In particular, it vanishes if $k < \mult_{\calRes}(T -\lambda\bft)$. Therefore $\lambda_0$ is a root of  $\phi^{\bft}_T(\lambda)$ of multiplicity greater or equal than $\mult_{\calRes}(T -\lambda_0\bft)$. This concludes the proof.
\end{proof}

\subsection{Resolution of the resultant hypersurface}

In this section we give an explicit resolution of singularities of the resultant hypersurface $\calRes$, and in particular we give a characterization of its smooth locus. The proof follows \cite[Example 1.2.3]{Dolg:ClassicalAlgGeometry}, which provides the similar characterization for the smooth locus of the discriminant hypersurface in $\bbP S^d W^*$. We expect this construction to be classical, but we could not locate an explicit proof in the literature, hence we provide a full proof for completeness.

Fix a basis $v_0 \vvirg v_n$ of $V$ and let $T=\sum_{i=0}^n v_i\ot f_i$ for some $f_0 \vvirg f_n\in S^{d-1}W^*$. Denote by 
\[
\Sing(T) = Z(f_0,\dots,f_n)\subseteq \bbP W
\]
 the scheme defined by $f_0 \vvirg f_n$. Notice that the definition of $\Sing(T)$ does not depend on the choice of the basis of $V$. When $T$ is symmetric, $f_0 \vvirg f_n$ generate the space of first order partial derivatives of $T$ and $\Sing(T)$ is the singular locus of the hypersurface defined by $T$ in $\bbP W$. 

\begin{proposition}\label{smooth locus of Res new}
The smooth locus of $\calRes \subseteq \bbP ( V \otimes S^{d-1} W^*)$ is the open set
\[
\calRes^{\smooth} = \{ T \in \bbP(V \otimes S^{d-1} W^*)\mid \Sing(T) \text{ is one reduced point}\}.
\]
\end{proposition}
\begin{proof}
The statement holds for $d=2$; see, e.g., \cite[Section II.2]{ArCoGrHa:Vol1}. If $d\ge 3$, consider the incidence correspondence 
    \[
    \tilde{\calRes} = \{ (T, w) \in \bbP (V \otimes S^{d-1} W^*) \times \bbP W\mid w \in  \Sing(T) \}. 
    \]
Note that $\tilde{\calRes}$ is a closed subvariety of $\bbP (V \otimes S^{d-1} W^*) \times \bbP W$. Let $\pi_\calRes : \tilde{\calRes} \to \bbP (V \otimes S^{d-1} W^*)$ and $\pi_W :  \tilde{\calRes} \to \bbP W$ be the two projections. The image of $\pi_{\calRes}$ is the hypersurface $\calRes \subseteq \bbP (V \otimes S^{d-1} W^*)$. Moreover $\tilde{\calRes}$ is smooth: indeed the map $\pi_W$ realizes $\tilde{\calRes}$ as a projective bundle over $\bbP W$. The fiber over $p\in \p W$ is the linear subspace $\bbP ( V \otimes \calI_{d-1} (p))$, where $\calI_{d-1} (p)\subseteq S^{d-1} W^*$ is the component of degree $d-1$ of the ideal of $p$.

By definition $\pi_\calRes^{-1}(T) = \{T\} \times \Sing(T)$. In particular, if $\Sing(T) = \{p\}$ is a single point, then $(T,p)$ is the unique point in the fiber $\pi_\calRes^{-1}(T)$. In this case a straightforward calculation, similar to the one of \cite[Example 1.2.3]{Dolg:ClassicalAlgGeometry} shows that $\pi_\calRes$ is an isomorphism in a Zariski neighborhood of $(T,p)$ and therefore $T$ is a smooth point of $\calRes$.

It remains to show that $T$ is a singular point of $\calRes$ if $\Sing(T)$ is not a reduced point. If $\Sing(T)$ is a $0$-dimensional scheme, then $\pi_\calRes^{-1}(T)$ is $0$-dimensional of length higher than the one of the generic fiber. In this case the Zariski Main Theorem \cite[Section III.11]{Hart:AlgGeom} guarantees that $T$ is a non-normal point of $\calRes$, and in particular it is singular.

If $\Sing(T)$ is positive-dimensional, then let $\{p_1,p_2\}$ be two points in $\Sing(T)$. Since $d\ge 3$, by the classical Trisecant Lemma \cite[Proposition 1.4.3]{Russo:GeometrySpecialProjVars}, the generic element  $T' \in V \otimes \calI_{d-1}(p_1,p_2)$ satisfies $\Sing(T') = \{ p_1,p_2\}$ and therefore it is in the singular locus of $\calRes$. Since $T \in V \otimes \calI_{d-1}(p_1,p_2)$ and the singular locus is Zariski closed, we conclude that $T$ is a singular point of $\calRes$. 
\end{proof}

\subsection{Simple eigenvalues}
In this section we generalize the classical result stating that a $(n+1)\times (n+1)$ matrix with distinct eigenvalues has $n+1$ linearly independent eigenvectors. We will prove that if a tensor $T$ has $D(n,d)$ distinct eigenvalues, then it has $D(n,d)$ distinct eigenvectors as well. This extends \cite[Lemma 6.1]{multiplicity}, which is the analogous statement when $\bft=\bfu_{n+1}$ and $T$ is generic in $V \otimes S^{d-1} W^*$.
\begin{theorem}\label{thm: distinct evals implies distinct evecs new}
Let $\bft,T \in V \otimes S^{d-1} W^*$ with $\res(\bft) \neq 0$. If the $\bft$-characteristic polynomial $\phi^\bft_T(\lambda)$ has $D(n,d)$ distinct roots, then 
\[
E^\bft_T = \{ w \in \bbP W\mid (\lambda,w) \text{ is a $\bft$-eigenpair of $T$ for some $\lambda$}\}
\]
consists of $D(n,d)$ reduced points.
\end{theorem}

\begin{proof}
 Note that $E^{\bft}_T = \bigcup_{\lambda \in \bbC} \Sing(T - \lambda \bft)$, where the union can be restricted to only the $D(n,d)$ values $\lambda$ satisfying $\res(T-\lambda \bft) = 0$. By \Cref{lemma: multiplicity bounded from below by mult new}, the multiplicity of $\lambda$ as a $\bft$-eigenvalue of $T$ is at least $\mult_{\calRes}(T-\lambda \bft)$, hence $\mult_{\calRes}(T-\lambda \bft) = 1$ for all eigenvalues $\lambda$ of $T$. In particular the tensor $T-\lambda \bft$ is a smooth point of $\calRes$, hence $\Sing( T-\lambda \bft )$ is a simple point $w(\lambda) \in \bbP W$ by \Cref{smooth locus of Res new}. Finally, if $\lambda_1 \neq \lambda_2$, by \Cref{remark: clearly} $w(\lambda_1) \neq w(\lambda_2)$ and this concludes the proof.
\end{proof}

It is natural to ask whether the converse of \Cref{thm: distinct evals implies distinct evecs new} holds.
\begin{remark}
If $d=2$, the converse of \Cref{thm: distinct evals implies distinct evecs new} is a consequence of the classical spectral theory of matrices. As observed in \Cref{rmk: for matrices we basically choose t equals u}, in this setting it is not restrictive to choose $\bft=\bfu_{n+1}$, and the corresponding choice of basis gives an identification $V \cong W$. If $T \in W \otimes W^*$ has $n+1$ distinct eigenvalues, then every eigenspace has dimension $1$, hence it corresponds to a simple point in $\bbP W$. If the eigenvalues are not distinct, either some eigenspace is of higher dimension, in which case $E^{\bft}_T$ contains infinitely many points, or $T$ is not diagonalizable, in which case $E^{\bft}_T$ is not reduced; see \cite[Corollary 4.9]{matrici}.
\end{remark}

However, when $d\ge 3$ the converse of \Cref{thm: distinct evals implies distinct evecs new} is not true.
\begin{example} If $d\ge 3$, then there exists a tensor $T \in \calRes$ such that $\Sing(T)$ consists of exactly two points. Let $T$ be generic among the tensors with this property. By Bertini's Theorem, the line $\langle T ,\bfu_{n+1}\rangle$ intersects $\calRes$ with multiplicity $2$ at $T$ and at $D(n,d)-2$ other distinct points which are smooth in $\calRes$. Each intersection point different from $T$ yields an eigenvector, as in the proof of \Cref{thm: distinct evals implies distinct evecs new}; the intersection at $T$ yields two eigenvectors, which are the two points of $\Sing(T)$, both having eigenvalue $0$. Therefore $T$ has $D(n,d)$ distinct eigenvectors but $\phi^\bft_{T}(\lambda)$ has a root of multiplicity two at $\lambda = 0$.
\end{example}

\subsection{Higher Hurwitz forms}
In light of \Cref{rmk: it is enough to prove that phi is defined everywhere}, we want to determine which tensors have only one eigenvalue. This property is closely related to the geometry of lines intersecting the resultant hypersurface with high multiplicity. To study such lines, we introduce the following objects.
\begin{definition}\label{def: hurwitz}
Let $X \subseteq \bbP ^N$ be an irreducible hypersurface. The \emph{variety of lines osculating at order $m$} is 
\[
\calH_X^{(m)} = \bar{\bigl\{ L \in \Gr(2, \bbC^{N+1}) \mid L \not \subseteq X, L \cap X \text{ has a component of degree at least $m$}\bigr\}}.
\]
Here $\Gr(2, \bbC^{N+1})$ denotes the Grassmannian of $2$-planes in $\bbC^{N+1}$, or equivalently the Grassmannian of lines in $\bbP^N$. For $D = \deg(X)$, write $\calH_X = \calH_X^{(\deg X)}$.
\end{definition}
\revision{The variety $\calH_X^{(2)}$ is a version of the Hurwitz form of $X$, as defined in \cite{Sturm:HurwitzForm}}. In general $\calH_X^{(2)}$ is reducible\footnote{\revision{In \cite[Theorem 1.1]{Sturm:HurwitzForm} the Hurwitz form is claimed to be irreducible. This is not the case in general, see \Cref{example: nodal cubic}. There are natural technical adjustments to make the Hurwitz form irreducible; these technicalities do not invalidate the results of \cite{Sturm:HurwitzForm}.}}, and one of its components is a hypersurface in the Grassmannian $\Gr(2, \bbC^{N+1})$: this hypersurface is the closure of the set of lines with the property that the $0$-dimensional component of degree $2$ in $L \cap X$ is supported at a smooth point of $X$. This component is the classical Hurwitz form studied in the context of associated varieties \cite[Section 3.2.E]{GKZ:DiscResMultDet}. If $m = 3$, then one of the components of $\calH_X^{(3)}$ is the variety of lines that are tangent to $X$ at a flex point. More generally, classical geometric results on the existence of flexes of high order on certain hypersurfaces can be phrased in terms of non-emptiness of $\calH_X^{(m)}$.

It is clear that $\calH_X^{(m)} \supseteq \calH_X^{(m+1)}$ for every $m$, and that $\calH_X^{(m)} = \emptyset$ whenever $m \geq \deg X + 1$. If $p \in X$ is a point of multiplicity $m$ then a generic line through $p$ is an element of $\calH^{(m)}_X$ and a generic line through $p$ and contained in $TC_p X$ is an element of $\calH^{(m+1)}_X$.
\begin{example}\label{example: nodal cubic}
    Let $X \subseteq \mathbb{P}^2$ be the nodal cubic curve defined by \revision{$f=x_1^2 x_0 - x_2^2(x_0-x_2)$}. Since $\Gr(2,3) \cong {\bbP^2}^\vee$, we have that $\calH^{(m)}_X$ are subvarieties of the dual projective space. 
The variety $\calH^{(2)}_X$ is a union of two components: the dual variety of $X$, which is a curve of degree $4$, and the pencil of lines through the node of $X$. This can be determined via an explicit computation as follows. We introduce new variables $u_i,v_i$, with $i\in\{0,1,2\}$, and compute the discriminant of the bivariate polynomial $g(s,t)=f(su_0+tv_0,su_1+tv_1,su_2+tv_2)$. This discriminant vanishes precisely when the projective line parametrized by $(su_0+tv_0,su_1+tv_1,su_2+tv_2)$ does not intersect $X$ in three reduced points. This can be computed by saturating the ideal generated by $g$ and its partial derivatives, and then eliminating the variables $s$ and $t$. In the Pl\"{u}cker coordinates $q_{ij}=u_iv_j-u_jv_i$, we have
\[
\revision{        I_{\calH^{(2)}_X} = \langle q_{12}^2\rangle \cap
        \langle 4q_{01}^4-8q_{01}^2q_{02}^2+4q_{02}^4+4q_{01}^3q_{12}-36q_{01}q_{02}^2q_{12}-27q_{02}^2q_{12}^2\rangle.
}\]
The first component defines the pencil of lines through the node $(1:0:0)$; the second is the dual variety of $X$, consisting of all lines which are tangent to $X$ at some point.

The variety $\calH^{(3)}_X$ consists of five distinct points, corresponding via duality to the tangent lines at the three flexes of $X$ and to the two lines tangent to each branch passing through the node. The explicit computation can be performed similarly as before. The defining ideal of $\calH_X^{(3)}$ is 
\[
\revision{   I_{\calH^{(3)}_X} = \langle q_{01},q_{02}\rangle  \cap \langle 8q_{01}+9q_{12}, 3q_{02}^2+ q_{01}^2\rangle \cap  \langle q_{12}, q_{01}-q_{02}\rangle \cap \langle q_{12}, q_{01}+q_{02}\rangle.}
\]
The first and the second components define the three lines tangent to the flexes of $X$, the third and fourth components define the two tangent lines at the node.
\end{example}

We expect the varieties $\calH_X^{(m)}$ to be hard to study in general. In \Cref{sec: forme binarie} and \Cref{sec: cubiche piane} we will prove that $\calH_\calDisc  = \emptyset$ in the case of binary forms and plane cubics, respectively. The following result shows that indeed $\calH_\calRes$ and $\calH_\calDisc$ control the non-emptiness of the indeterminacy locus of the characteristic polynomial maps $\Phi^\p_\bft$ and $\Phi^\p_{\bft,\sym}$.

\begin{lemma}\label{lemma: only 0 iff H empty}
Let $\bft,T \in V \otimes S^{d-1} W^*$ be tensors with $\res(\bft) \neq 0$. Then $\phi^{\bft}_T(\lambda) = \lambda^{D(n,d)}$ if and only if $T \in \calRes$ and the line $\langle T , \bft \rangle \subseteq \bbP ( V \otimes S^{d-1}W^*)$ is an element of $\calH_\calRes$.
\end{lemma}
\begin{proof}
Parametrize the line $L = \langle T,\bft\rangle$ as $L_\lambda = T -\lambda \bft$; then $\phi^\bft_T(\lambda) = \res|_L$. By hypothesis the only root of the characteristic polynomial is $0$, which guarantees that the $D(n,d)$ intersection points of $\calRes$ and $L$ all coincide with $L_\lambda|_{\lambda = 0} = T$. By definition, this is equivalent to the condition $L \in \calH_\calRes$.
\end{proof}

\section{Characteristic polynomial of binary forms}\label{sec: forme binarie}

In this section we study \Cref{conj:fibre symmetric} in the case $n=1$. In \Cref{thm: forme binarie} we show that there are no non-zero polynomials such that all of their eigenvalues are equal to zero. In \Cref{thm: fibre 0 gen t} we deduce the finiteness of every fiber of $\Phi_{\bft,\sym}$ when $n=1$. 

\subsection{Resultant of binary tensors}

In the case of binary forms, the resultant polynomial has a determinantal formula, dating back to \cite{Sylv:TheorySyzRelations}. Given two nonnegative integers $d$ and $e$, let $\mu : S^d W^* \otimes S^e W^* \to S^{d+e} W^*$ be the natural multiplication map, defined by $\mu(f \otimes g) = fg$ and extended linearly.

Assume that $\dim V = \dim W = 2$. Let $T \in V \otimes S^{d-1} W^*$ be a tensor and identify $T$ with the linear map $T : V^* \to S^{d-1}W^*$ that it defines. Let $\Sigma_T: S^{d-2} W^* \otimes V^* \to S^{2d-3} W^*$ be the composition 
\[
S^{d-2} W^* \otimes V^* \xto{\id_{S^{d-2} W^*}  \otimes T} S^{d-2} W^* \otimes S^{d-1}W^* \xto{\mu} S^{2d-3} W^*.
\]
Explicitly, if $T = v_0 \otimes f_0 + v_1 \otimes f_1$ for a basis $\{v_0,v_1\}$ of $V$ with dual basis $\{\alpha_0,\alpha_1\}$ of $V^*$, then 
\[
\Sigma_T( g_0 \otimes \alpha_0 + g_1 \otimes \alpha_1) = g_0 f_0 + g_1 f_1.
\]

Notice that $\dim (S^{d-2} W^* \otimes V^*) = \dim ( S^{2d-3} W^* ) = 2d-2 = D(1,d)$. The determinant of $\Sigma_T$ is a polynomial of degree $2d-2$ in the coefficients of $T$; since the map $\Sigma_T$ is equivariant for the action of $\GL(V) \times \GL(W)$, this polynomial is an invariant. Note that $\ker( \Sigma_T) \neq 0$ if and only if the ideal $(f_0,f_1)$ defining $\Sing(T)$ has syzygies in degree $d-1$: this happens exactly when $\Sing(T)\neq \emptyset$, that is $T \in \calRes$. We have therefore the next result; we refer to \cite[Section 12.1.A]{GKZ:DiscResMultDet} for details.

\begin{lemma}\label{coroll: degree determinant}
The determinant of $\Sigma_T$, as a polynomial in the coefficients of $T$, is the resultant polynomial $\res$. Its restriction to the symmetric subspace $S^d W^*$ is the discriminant polynomial $\disc$.
\end{lemma}

\subsection{Tangent cones of the discriminant}
The proof of \Cref{thm: forme binarie} relies on a structural result on the tangent cones to the hypersurface $\calDisc$, which is already contained in \cite{Hilb:UberSingDisc}. 

We introduce the following notation. Given $w \in \bbP W$, let 
\[
H_w = \{ f \in S^d W^* \mid f(w) = 0\} = \calI_d(w)
\]
be  the hyperplane in $S^d W^*$ of forms vanishing at $w$. If $w = \{ \ell = 0\}$, then $H_w = \ell \cdot S^{d-1} W^*$. Given a hypersurface $Y\subseteq \bbP^N$ defined by a homogeneous polynomial $G$, denote by $Y^{\langle m \rangle}$ the scheme of pure codimension $1$ and equation $G^m$. 

\begin{lemma}\label{lemma: tangent cone to disc}
Let $\calDisc \subseteq \bbP ( S^{d} W^*)$ be the discriminant hypersurface of binary forms of degree $d$. Let $f \in S^d W^*$ and write $f = \prod_1^s \ell_i ^{e_i}$ for pairwise linearly independent linear forms $\ell_i$ and exponents $e_i$ with $d = e_1 + \cdots + e_s$ and $e_1 \geq \cdots \geq e_s > 0$. Then
    \begin{enumerate}[label = (\roman*), leftmargin= *]
        \item $f \in \calDisc$ if and only if $e_1 \geq 2$, or equivalently $s <d$;
        \item $f$ is a smooth point of $\calDisc$ if and only if $s = d-1$, $e_1 = 2$ and $e_j = 1$ for every $j \in\{2 \vvirg d-1\}$;
        \item the multiplicity of $f$ in $\calDisc$ is $\sum_1^s (e_j-1)$;
        \item the tangent cone $TC_f \calDisc$ is the scheme-theoretic union $H_{w_1}^{\langle e_1-1\rangle} \cup \cdots \cup H_{w_s}^{\langle e_s-1 \rangle}$, where $w_j = \{ \ell_j = 0\}$.
    \end{enumerate}
\end{lemma}
\begin{proof}
    Statement (i) is clear. Statements (ii) and (iii) follow from statement (iv). We prove statement (iv) by computing the lowest order term in the expansion of the discriminant polynomial at the point $f = \ell_1^{e_1} \cdots \ell_s^{e_s}$.

Localize on a chart of $\bbP^1$, so that, in the local coordinate $(z,1) \in \bbP^1$, we have $\ell_j = (z-\alpha_j)$, $w_j = (\alpha_j,1)$ and $f(z,1)$ is a univariate polynomial of degree $d$ in $z$ which factors as 
\[
f(z,1) = (z - \alpha_1)^{e_1} \cdots (z - \alpha_s)^{e_s}.
\]
Let $g \in S^d W^*$ and consider $f_\eps(z,1) = f(z,1) + \eps g(z,1)$. We compute the lowest order term, as a function of $\eps$, of the univariate polynomial $\disc(f_\eps) \in \bbC[\eps]$. The vanishing of such term defines the tangent cone. Now, $f_\eps$ factors as a product of linear forms over the field of Puiseux series $\bbC\{\!\{\eps\}\!\}$, see \cite[Thm. 2.2.6]{Wall}. More precisely, one has 
\[
f_\eps = \prod_{\substack{i  = 1 \vvirg s \\ 
j = 1 \vvirg e_i}}
(z - \alpha_{ij}(\tau)) 
\]
where, setting $e = \mathrm{lcm}(e_1 \vvirg e_s)$, we have $\tau = \eps^{1/e}$ and for every $i,j$, the function $\alpha_{ij}(\tau)$ is a power series in $\tau$ with $\alpha_{ij}(0) = \alpha_i$. Set $\tau_i = \tau^{e/e_i} = \eps^{1/e_i}$. We show that $\alpha_{ij}(\tau) = \alpha_i + \tau_i \alpha_{ij}' + o(\tau_i)$. To see this, write $f = \ell_i^{e_i} \cdot h_i$ for a certain $h_i \in S^{d-e_i} W^*$ with $h_i(\alpha_i,1) \neq 0$. Moreover, write $\alpha_{ij}(\tau) = \alpha_i + \tau^{r_{ij}} \beta_{ij}(\tau)$ where $\beta_{ij}(\tau)$ is a power series in $\tau$ with $\beta_{ij}(0) \neq 0$. Then, 
\begin{align*}
0 = f_\eps(\alpha_{ij}(\tau),1) = &f(\alpha_{ij}(\tau),1) + \tau^e g(\alpha_{ij}(\tau),1) =  \\
&\tau^{r_{ij}e_i} \beta_{ij}(\tau)^{e_i} h_i(\alpha_{ij}(\tau),1) + \tau^e g(\alpha_{ij}(\tau),1) .
\end{align*}
This expression vanishes for every $\tau$. In particular, the lowest order term vanishes; this guarantees $e_i r_{ij} = e$, and therefore $\alpha_{ij}(\tau) = \alpha_i + \tau_i \alpha_{ij}' + o(\tau_i)$, where $\alpha_{ij}' = \beta_{ij}(0)$. Moreover, the vanishing of the lowest order term guarantees that
\begin{equation}\label{eqn: lowest order term in tau}
\alpha_{ij}'^{e_i} h_i(\alpha_i,1) +  g(\alpha_i,1) = 0.
\end{equation}
Now, we evaluate the discriminant $\disc(f_\eps)$ as a polynomial in $\tau$, 
 using the expression of the discriminant in terms of differences of the roots. Up to a nonzero scalar, we obtain 
\[
\disc (f_\eps) = \prod_{(i_1,j_1) \neq (i_2,j_2)} (\alpha_{i_1 j_1}(\tau) - \alpha_{i_2 j_2}(\tau)).
\]
The lowest order term in $\disc (f_\eps)$ as a function of $\tau$ arises as the product of the lowest order terms in every factor. If $i_1 \neq i_2$, the lowest order term is simply $\alpha_{i_1,j_1}(0) - \alpha_{i_2,j_2}(0) = \alpha_{i_1} - \alpha_{i_2}$, which has degree $0$ in $\tau$. If $i_1 = i_2$, the lowest order term has degree $1$ in $\tau_i$ and it is  $(\alpha_{i_1j_1}' - \alpha_{i_2,j_2}')\tau_i$; there are $e_{i_1}(e_{i_1}-1)$ such terms for every $i_1$. Using $\tau_i^{e_i} = \eps$, we obtain  
\[
\disc(f_\eps) = C \cdot \eps^{\sum_1^s(e_i-1)} \prod_{i=1}^s \prod_{1 \leq j_1 < j_2 \leq e_i} (\alpha_{ij_1}' - \alpha_{ij_2}')^2 + \text{ higher order terms}.
\]
We are left to evaluate the terms $\prod_{1 \leq j_1 < j_2 \leq e_i} (\alpha_{ij_1}' - \alpha_{ij_2}')^2$ for every fixed $i$. This is the square of the $e_i \times e_i$ Vandermonde determinant defined by $\alpha_{i1}' \vvirg \alpha_{ie_i}'$. From \eqref{eqn: lowest order term in tau} we see that, up to reordering, we have $\alpha_{ij}' = \alpha_{i1}' \omega^j$, where $\omega$ is a primitive $e_i$-th root of $1$. Therefore, up to rescaling by a root of $1$, we have
\[
\prod_{1 \leq j_1 < j_2 \leq e_i} (\alpha_{ij_1}' - \alpha_{ij_2}')^2 = \alpha_{i1}'^{e_i(e_i-1)} = \left[-\frac{g(\alpha_i,1)}{h_i(\alpha_i,1)}\right]^{e_i-1}.
\]
Up to renormalizing by the scalars $h_i(\alpha_i,1)$, we obtain that the lowest order term in $\disc(f_\eps)$ has degree $\sum_1^s (e_i-1)$ in $\eps$ and it is $\prod_{i=1}^s g(\alpha_i,1)^{e_i-1}$. The vanishing of this lowest order term is equivalent to $g \in H_{w_1}^{\langle e_1-1\rangle} \cup \cdots \cup H_{w_s}^{\langle e_s-1 \rangle}$, which concludes the proof.
\end{proof}

\subsection{Eigenvalues of binary forms}
In this section, we prove that the variety $\calH_\calDisc$, defined in \Cref{def: hurwitz}, is empty in the case of binary forms, see \Cref{thm: forme binarie}. As a consequence the map \revision{$\Phi^\p_{\bft,\sym}$ has no indeterminacy points} for any choice of $\bft \in S^d W^*$ with $\disc(\bft) \neq 0$. This yields the proof of \Cref{thm: main} in the case of binary forms. 

The proof of \Cref{thm: forme binarie} is by induction on $d$. We record separately the base case of the induction.
\begin{lemma}\label{lemma: bincubics}
    If $d = 3$ and $n = 1$, then $\calH_\calDisc = \emptyset$. 
\end{lemma}
\begin{proof}
Let $f \in \calDisc \subseteq \bbP S^3 W^*$ and let $g \in \bbP S^3 W^*$. Define $L = \langle f, g\rangle$ and suppose $L \not\subseteq \calDisc$; without loss of generality, assume that $g \notin \calDisc$, and $\disc(g) = 1$. By \Cref{coroll: degree determinant}, $\deg(\calDisc) = 4$. We are going to show that the intersection multiplicity of $L$ and $\calDisc$ at $f$ is strictly smaller than $4$: more precisely, parametrizing $L$ as $f_t = f + tg$, we show that $\disc(f_t) \neq t^4$. 

It is a classical fact \cite{Sylv:PrinciplesCalculusForms} that the space $\bbP S^3 W^*$ has only three orbits for the action of $\GL(W)$ and the hypersurface $\calDisc$ is the union of two of them: hence $f$ can be normalized to be $f = x_0^3$ or $f = x_0^2x_1$. Write $g=a_0x_0^3+3 a_1 x_0^2x_1+3 a_2 x_0x_1^2+ a_3 x_1^3$ for certain coefficients $a_0 \vvirg a_3$. If $g \notin TC_f \calDisc$ then the intersection multiplicity equals $\mult_\calDisc(f)$: this is $2$ if $f= x_0^3$ and $1$ if $f = x_0^2x_1$. Hence we may assume $g \in TC_f \calDisc$, which by \Cref{lemma: tangent cone to disc} is equivalent to $a_3 = 0$.

If $f=x_0^3$, then $\disc (f_t) = t^3 (\disc(g) t - 4 a_2^3)$. Hence $\disc(f_t) = t^4$ if and only if $a_2 = 0$, which implies $g = x_0^2 \ell$ for some $\ell \in W^*$, hence $g \in \calDisc$ which is a contradiction.

If $f=x_0^2x_1$, then $\disc(f_t)=t^2(\disc(g) t^2 + 2 a_1a_2^2 t - \frac{1}{3} a_2^2)$. Therefore if $\disc(f_t) = t^4$, necessarily $a_2 = 0$, which again implies $g = x_0^2 \ell$ for some $\ell \in W^*$.  Again, $g\in \calDisc$, and we obtain a contradiction.
    \end{proof}

\begin{theorem}\label{thm: forme binarie}
If $n =1$ and $d \geq 3$, then $\calH_\calDisc=\emptyset$.
\end{theorem}
\begin{proof}
Let $f \in \calDisc$. We show that there are no elements $L \in \calH_\calDisc$ such that $f \in L$. Let $L \in \Gr(2 , S^d W^*)$ be a line through $f$. Suppose by contradiction that $L \in \calH_{\calDisc}$; by assumption $L$ intersects $\calDisc$ only at $f$. Fix $g \in L\setminus\{f\}$, and consider the parameterization of $L$ given by $f + t g$; since $L$ intersects $\calDisc$ only at $f$, the binary form $f  + tg$ has distinct factors for every $t \neq 0$. We proceed by induction on $d$ to show that the line $L$ does not exist. If $d = 3$, \Cref{lemma: bincubics} shows that such a line does not exist. 
    
    Let $d \geq 4$. If $L \not\subseteq TC_f\calDisc$, then the intersection multiplicity of $\calDisc$ and $L$ at $f$ is the multiplicity of $f$ in $\calDisc$; by \Cref{lemma: tangent cone to disc}(iv), this is bounded above by $d-1$, hence it is not $D(1,d) = 2d-2$. Hence, assume $L \subseteq TC_f \calDisc$. By \Cref{lemma: tangent cone to disc}(iv), there is a factor $\ell$ of $f$ which is also a factor of $g$: write $f + t g = \ell\cdot (f' + t g')$. 

Since the factors of $f' + t g'$ are factors of $f + tg$, the form $f' + t g' \in S^{d-1} W^*$ must have distinct factors for $t \neq 0$. This defines a line $L'$ in $\bbP S^{d-1} W^*$ which intersects the discriminant of the forms of degree $d-1$ in exactly one point $f'$. The induction hypothesis guarantees that such a line does not exist. This provides a contradiction and concludes the proof, showing that the line $L \in \calH_\calDisc$ cannot exist.
\end{proof}

The main result of this section shows that the map $\Phi_{\bft,\sym}$ has finite fibers, which in turn yields the proof of \Cref{thm: main}.

\begin{theorem}\label{thm: fibre 0 gen t}
Let $\dim W=2$ and $\bft\in S^d W^\ast$ with $\disc(\bft)\neq 0$. Then all fibers of $
    \Phi_{\bft,\sym}
    $
    are finite.
\end{theorem}
\begin{proof}
    \Cref{thm: forme binarie} shows that the fiber of the affine map $\Phi_{\bft,\sym}$ at $\lambda^{D(1,d)}$ is given set-theoretically by $\{0\}$. Therefore the projectivization
    \[
    \Phi^\p_{\bft,\sym} : \bbP S^d W^* \to \bbP ( \bfrho_{D(1,d)})
    \]
    is a well-defined morphism of projective spaces. By \Cref{prop: regular implies finite}, we deduce that all fibers of $\Phi_{\bft,\sym}$ are finite.
\end{proof}
Applying \Cref{thm: fibre 0 gen t} to the case $\bft = \bfu_2$, we obtain \Cref{thm: main} for the case of binary forms. We conclude this section with an interesting observation for the case $(n,d)=(1,3)$ of binary cubics.
\begin{remark}\label{remark: binary cubic surjective}
    For $d = 3$ the map $\Phi^\p_{\bft,\sym} : \bbP S^3 W^* \to \bbP (\bfrho_{4})$ is projective and dominant, hence surjective. As a consequence the affine map $\Phi_{\bft,\sym} : S^3 W^* \to \bbC[\lambda]^{\mon}_4$ is surjective as well, and every univariate polynomial of degree $4$ is the characteristic polynomial of a binary form of degree $3$.
    
    Given a generic univariate polynomial $\phi(\lambda)$ of degree $4$, the condition $\res(T - \lambda \bft) = \phi(\lambda)$ defines a \revision{polynomial system} of four equations in four unknowns: the four unknowns are the coefficients of $T \in S^3 W^*$ and  the four equations are of degree $1,2,3,4$. Since \revision{$\Phi^\p_{\bft,\sym}$} is well-defined, the polynomial system has no solutions at infinity; this guarantees that it has $1 \cdot 2 \cdot 3 \cdot 4 = 24$ solutions, which, by genericity, are distinct. This shows that the fiber $\Phi_{\bft,\sym}^{-1}(\phi(\lambda))$ has degree $24$ for every monic polynomial $\phi(\lambda)$ and it consists of $24$ distinct points when $\phi(\lambda)$ is generic. 
\end{remark}

\section{Characteristic polynomial of plane cubics}\label{sec: cubiche piane}
\label{section: proof of main for plane curves}
In this section, we prove \Cref{thm: main} in the case $(n,d)=(2,3)$, see \Cref{thm: fibre 0 gen t n = 2 d = 3}. The proof is similar to that of binary forms. We show in \Cref{thm: binary cubics} that the variety $\calH_{\calDisc}$ defined in \Cref{def: hurwitz} is empty, which in turn guarantees that there is no ternary cubic having all eigenvalues equal to $0$. This implies that the map $\Phi^\p_{\bft,\sym}$ is well defined everywhere, hence finite by \Cref{prop: regular implies finite} and this allows us to conclude. 

Fix $\dim W = 3$. It is a classical fact that the discriminant in this case is a hypersurface $\calDisc \subseteq \bbP S^3 W^*$ of degree $12$. Moreover, this hypersurface has a stratification into eight orbits for the action of the group $\GL(W)$. This stratification was known classically \cite{vinberg}, and a modern description is given in \cite{KoganMaza:ComputationCanonicalFormsTernaryCubics}. We are interested in a description of the tangent cones to $\calDisc$.

Similarly to \Cref{sec: forme binarie}, given $w \in \bbP W = \bbP^2$, write $H_w  = \{ f \in S^3 W^* \mid f (w) = 0\}$ for the hyperplane of forms vanishing at $w$; as before, we use angular brackets for multiplicities of a hypersurface.

For every orbit, we record a description of the tangent cone to $\calDisc$ at an element of the orbit, together with its multiplicity. The proof is purely computational; the code to perform it is provided in \cite{BryGesSt:PencilsCubics8Pts} and, for the reader's convenience, at 

\centerline {\href{https://fulges.github.io/code/GGTV_Eigenvalues/index.html}{\texttt{https://fulges.github.io/code/GGTV\_Eigenvalues}.}}
\begin{itemize}[leftmargin=*]
 \item Nodal cubic: Let $f$ be an irreducible cubic with a simple nodal singularity at $p$. Then $f$ is in the orbit of $x_0x_1^2 - x_2^2(x_2-x_0)$ and
 \[
  TC_{[f]} \calDisc = H_p .
 \]
In particular, $\mult_\calDisc(f) = 1$ and \revision{the orbit of the nodal cubic} is the locus of smooth points of $\calDisc$.
 \item Cuspidal cubic: Let $f$ be an irreducible cubic with a cuspidal singularity at $p$. Then $f$ is in the orbit of $x_0 x_1^2- x_2^3$ and
 \[
TC_{[f]}\calDisc = H_p^{\langle 2 \rangle}.
\]
In particular, $\mult_\calDisc(f) = 2$.
 \item Conic and secant line: Let $f = \ell \cdot q$ with $q \in S^2 W^*$ and $\ell \in S^1 W^*$ such that $V(\ell,q)$ consists of two distinct points. Then $f$ is in the orbit of $x_0(x_0^2+x_1x_2)$. If $\{p_1,p_2\} = V(\ell,q)$ then $f$ is singular at $p_1,p_2$ and 
 \[
  TC_{[f]}\calDisc = H_{p_1} \cup H_{p_2} .
 \]
In particular $\mult_\calDisc(f) = 2$.
\item Conic and tangent line: Let $f= \ell \cdot q$ with $q \in S^2 W^*$ and $\ell \in S^1 W^*$ such that $V(\ell,q)$ is a non-reduced $0$-dimensional scheme of length $2$ supported at a point $p$. Then $f$ is in the orbit of $x_0 (x_0x_1 + x_2^2)$.  If $p$ is the support point of the $0$-dimensional scheme $V(\ell,q)$, then $f$ is singular at $p$ and
 \[
TC_{[f]}\calDisc = H_p^{\langle 3 \rangle}.
\]
In particular, $\mult_\calDisc(f) = 3$. 
\item Triangle: Let $f = \ell_0\ell_1\ell_2$ with $\ell_0,\ell_1,\ell_2 \in S^1 W^*$ linearly independent. Then $f$ is in the orbit of $x_0x_1x_2$. If $p_0,p_1,p_2$ are the three pairwise intersections of the three lines, then $f$ is singular at $p_0,p_1,p_2$ and 
 \[
  TC_{[f]}\calDisc = H_{p_0} \cup H_{p_1} \cup H_{p_2} .
 \]
In particular $\mult_\calDisc(f) = 3$.

\item Asterisk: Let $f = \ell_0\ell_1\ell_2$ with $\ell_0,\ell_1,\ell_2 \in S^1 W^*$ pairwise linearly independent but with $\dim \langle \ell_0,\ell_1,\ell_2\rangle = 2$. Then $f$ is in the orbit of $x_0 x_1(x_0+x_1)$. The variety $V(\ell_0,\ell_1,\ell_2)$ consists of a single point $p$, the curve $f$ is a cone over $p$ and 
\[
TC_{[f]}\calDisc =  H_p^{\langle 4 \rangle}.
\]
In particular $\mult_\calDisc(f) = 4$.
\item Line and double line: Let $f = \ell_0^2\ell_1$ with $\ell_0,\ell_1\in S^1 W^*$ linearly independent. Then $f$ is in the orbit of $x_0^2x_1$. Moreover, $f$ is singular at every point of $V(\ell_0)$ and it is a cone over $p = V(\ell_0,\ell_1)$. We have 
\[
TC_{[f]}\calDisc =  H_p^{\langle 2 \rangle}  \cup \calDisc'_{\ell_0}
\]
where $\calDisc'_{\ell_0} = \{ g \in \bbP S^3 W^* \mid g|_{\ell_0 = 0} \text{ is singular}\}$. Notice $\deg \calDisc'_{\ell_0} = 4$ because it is a cone over the variety of singular binary cubics. In particular $\mult_\calDisc(f) = 2 + 4 = 6$.
\item Triple line: Let $f = \ell^3$. Then $f$ is in the orbit of $x_0^3$. Moreover, $f$ is singular at every point of $V(\ell)$. We have
\[
TC_{[f]}\calDisc = {\calDisc'}_{\ell}^{\langle 2 \rangle}.
\]
Since $\deg (\calDisc'_\ell) = 4$, we obtain $\mult_\calDisc(f) = 8$.
 \end{itemize}

\begin{theorem}\label{thm: binary cubics}
    If $n=2$ and $d=3$, then $\calH_{\calDisc}=\emptyset$.
\end{theorem}
\begin{proof}
Let $f\in\calDisc$ and let $L\in\Gr(2,S^3 W^*)$ be a line through $f$; by definition, $L \in \calH_{\calDisc}$ if and only if $L$ intersects $\calDisc$ with multiplicity $12$ at $f$. If $L\not\subseteq TC_f\calDisc$ then the intersection multiplicity of $\calDisc$ and $L$ at $f$ equals $\mult_{\calDisc}(f)$. From the classification of the eight $\GL(W)$-orbits, we have $\mult_{\calDisc}(f) \leq 8$, so $L \notin \calH_{\calDisc}$.  
    
    Assume that $L\subseteq TC_f\calDisc$. For each of the eight $\GL(W)$-orbits let $f$ be an element in such an orbit and $g = \sum_{i+j+k=3}a_{ijk}x_0^i x_1^j x_2^k$ be some element of $L$ with $g \notin \calDisc$. Parametrize $L$ as $f_t=f+tg$, and write $\disc(f_t)=\sum_{i=0}^{12}c_it^i$ where each coefficient $c_k$ is a homogeneous polynomial of degree $k$ in the variables $a_{ijk}$. The assumption $L\subseteq TC_f\calDisc$ guarantees $g\in TC_f\calDisc$, which imposes certain polynomial conditions on the coefficients of $g$. The condition $L \in \calH_{\calDisc}$ guarantees $\disc(f_t)=t^{12}$, which is equivalent to the fact that the coefficients $c_0,\dots,c_{11}$ vanish. Using the software \texttt{Macaulay2} \cite{M2}, we verify that for each orbit the ideal $J=\sqrt{\langle c_0,\dots,c_{11}\rangle}$ contains the polynomial $\disc(g)$. In other words, imposing that $\disc(f_t)=t^{12}$ implies that  $g\in\calDisc$, which is a contradiction.

The code to run this computation is available at 

\centerline{
\href{https://fulges.github.io/code/GGTV_Eigenvalues/index.html}{\texttt{https://fulges.github.io/code/GGTV\_Eigenvalues}}.}

The calculation requires some technical speed-up. In particular, for each component $\Theta$ of $TC_f\calDisc$, we manually impose $g \in \Theta$ and rather than proving $\disc(g) \in J$ we prove that a suitable power of $\disc(g)$ belongs to $J + I_\Theta$, where $I_\Theta$ is the defining ideal of $\Theta$. 
    
    For the last two orbits there is a nonlinear component $\Theta$ in $TC_f\calDisc$, which makes the computation harder. In these cases, the problem can be further reduced. The condition $g\in \calDisc'_{x_0=0}$ implies that $g|_{x_0=0}$ is a singular binary cubic in the variables $x_1$ and $x_2$. When $f=x_0^2x_1$ is the chosen representative of the seventh orbit, we may normalize $g$ via the action of the isotropy group of $f$ in $\GL(\mathbb{C}^3)$. In this case, we may assume $g|_{x_0=0}$ is one of the following $5$ binary cubics: $x_1^3$, $x_2^3$, $x_2^2(x_1+x_2)$, $x_1^2x_2$, $x_1x_2^2$.  Similarly, when $f=x_0^3$, $g|_{x_0=0}$ can be normalized to be either $x_1^3$ or $x_1^2x_2$. 
\end{proof}

We use the result above to show that the fibers of $\Phi_{\bft,\sym}$ are finite in the case of plane cubics.
\begin{theorem}\label{thm: fibre 0 gen t n = 2 d = 3}
    Let $\bft\in S^3 W^\ast$ with $\disc(\bft)\neq 0$. Then all fibers of $
    \Phi_{\bft,\sym}
    $
    are finite.
\end{theorem}
\begin{proof}
    We proceed as in the proof of \Cref{thm: fibre 0 gen t}.
    Let $T\in S^3 W^\ast$. By \Cref{lemma: only 0 iff H empty},  $\Phi_{\bft,\sym}(T)=\lambda^{12}$ if and only if $\langle T,\bft\rangle\in\calH_{\calDisc}$. By \Cref{thm: binary cubics}, $\calH_{\calDisc}=\emptyset$, and so the fiber of $\lambda^{12}$ consists only of the zero polynomial. Hence the map 
    \[
        \Phi^\p_{\bft,\sym} : \bbP S^3 W^* \to \bbP ( \bfrho_{12})
    \]
    is a well-defined projective morphism. All its fibers are finite by \Cref{prop: regular implies finite}.
\end{proof}
As in the previous section, we obtain \Cref{thm: main} for the case of plane cubics by applying \Cref{thm: fibre 0 gen t n = 2 d = 3} in the case $\bft = \bfu_3$.

{
\bibliographystyle{alphaurl}
\bibliography{referenzeautovalori}
}
\end{document}